\documentclass[11pt,a4paper]{amsart} 
\usepackage[all]{xy}
\usepackage{amssymb}
\usepackage{braket}
\SelectTips{cm}{}
\calclayout \makeatletter
\def\serieslogo@{} 
\def\@setcopyright{} 
\makeatother

\title[The finitistic dimension of algebras with a directed stratification]{The finitistic dimension of algebras with a directed stratification}

\author{Karsten Dietrich}
\address{Karsten Dietrich\\ Fakult\"at f\"ur Mathematik\\
Universit\"at Bielefeld\\ D-33501 Bielefeld\\ Germany.}
\email{karsten.dietrich@gmx.net}


\newtheorem{lem}{Lemma}[section]
\newtheorem{prop}[lem]{Proposition}
\newtheorem{cor}[lem]{Corollary}
\newtheorem{thm}[lem]{Theorem}

\theoremstyle{definition}
\newtheorem{exm}[lem]{Example}
\newtheorem{Rem}[lem]{Remark}
\newtheorem{defn}[lem]{Definition}

\numberwithin{equation}{section}
\newtheorem*{ackn}{Acknowledegement}

\renewcommand{\mod}{\operatorname{mod}\nolimits}
\newcommand{\rep}{\operatorname{rep}\nolimits}

\newcommand{\Fun}{\operatorname{Fun}\nolimits}
\newcommand{\Aut}{\operatorname{Aut}\nolimits}
\newcommand{\Stab}{\operatorname{Stab}\nolimits}
\newcommand{\findim}{\operatorname{fin.dim}\nolimits}
\newcommand{\gldim}{\operatorname{gl.dim}\nolimits}
\newcommand{\repdim}{\operatorname{rep.dim}\nolimits}
\newcommand{\Ob}{\operatorname{Ob}\nolimits}

\newcommand{\rad}{\operatorname{rad}\nolimits}

\newcommand{\pd}{\operatorname{proj.dim}\nolimits}
\newcommand{\id}{\operatorname{id}\nolimits}

\newcommand{\Mod}{\operatorname{Mod}\nolimits}
\newcommand{\Findim}{\operatorname{Fin.dim}\nolimits}
\renewcommand{\mod}{\operatorname{mod}\nolimits}

\newcommand{\Hom}{\operatorname{Hom}\nolimits}

\renewcommand{\Im}{\operatorname{Im}\nolimits}

\newcommand{\Ext}{\operatorname{Ext}\nolimits}
\newcommand{\Tor}{\operatorname{Tor}\nolimits}

\def\A{{\mathcal A}}
\def\B{{\mathcal B}}
\def\C{{\mathcal C}}

\def\M{{\mathcal M}}
\def\N{{\mathcal N}}
\def\P{{\mathcal P}}

\begin{document}

\begin{abstract}
	We introduce the notion of a directed stratification for a finite-dimensional algebra. For algebras that admit such a stratification we characterise the projective resolutions of finitely generated modules and obtain a result for the finitistic dimension, which is an inductive version of a result of Fossum, Griffith and Reiten. With the developed techniques, which are adopted from the theory of EI-category algebras, we gain deeper insight in the combinatorial nature of this result. A characterisation of algebras which do not admit a directed stratification is given in terms of the Ext-quiver.
	\end{abstract}

\maketitle
\setcounter{tocdepth}{1}
\tableofcontents

\section{Introduction}
The finitistic dimensions of a ring $\Lambda$ provide a measure for the complexity of the module category of $\Lambda$. They are defined as
\begin{eqnarray*}
	\findim (\Lambda) &=& \sup \Set{ \pd M | M \in \mod (\Lambda), \pd M < \infty } \\
    \Findim (\Lambda) &=& \sup \Set{ \pd M | M \in \Mod (\Lambda), \pd M < \infty }. 
	\end{eqnarray*}
	There are at least two canonical questions that arise in studying these invariants, namely: Are these two dimensions finite for any ring $\Lambda$ and do they coincide? For noetherian rings both questions have to be answered in the negative, but for finite-dimensional algebras there is no counterexample up to now. In 1960 Bass published the two questions for finite-dimensional algebras as ``problems''   and they are nowadays known as the finitistic dimension conjectures. 
	 
The little finitistic dimension, $\findim$,  is known to be finite for certain classes of algebras, for example for algebras with representation dimension at most $3$, monomial algebras or algebras with radical cube zero. One may consult \cite{Birge_tale} for a survey on this conjecture and other homological conjectures (not including the result of Igusa and Todorov from \cite{IT} concerning the relation of the representation dimension and the finitistic dimension).

First, we will introduce the notion of a directed stratification for a finite-dimensional algebra. This definition is inspired by the study of EI-category algebras, where one knows that the finitistic dimension is always finite, see \cite{Ich} for a survey on the finitistic dimension of EI-category algebras. The proof of this finiteness of $\findim$ for EI-category algebras reduces the problem to the finitistic dimension of the group algebras of the automorphism groups. This concept will be generalised to algebras with a directed stratification in the second and third section, where we show that finiteness of the finitistic dimension of such an algebra only depends on the finitistic dimensions of the strata.  We want to emphasize that this reduction technique has already been obtained by Fossum, Griffith and Reiten and Fuller, Saorin as a special case of a theorem in the context of trivial extensions of abelian categories, compare \cite{FGR,FullerSaorin}. Nevertheless, the approach we will present in terms of representations of a category which is associated to an algebra with a directed stratification gives a very convenient combinatorial description of the projective resolutions of the modules. We will also give a characterisation of algebras which are minimal in the sense, that they do not admit a non-trivial directed stratification in terms of their Gabriel quiver (Ext quiver). Finally, in the last section, we relate the mentioned reduction technique to other known results for the finitistic dimension, for example to results of Happel \cite{Happel}, Cline, Parshall and Scott \cite{CPS},\cite{CPS2},\cite{CPS3} and Huisgen-Zimmermann \cite{Birge_monomial}.
\begin{ackn}
This paper is based on parts of my PhD-thesis under the supervision of Henning Krause. I would like to thank him for his guidance and inspiring discussions on the topic.

The first version of this paper contained a proof of the finiteness of the finitistic dimension for EI-category algebras. I am grateful to Jesper Grodal for pointing out the relevance of Lueck's work \cite{Lueck} which already contains a proof of this fact.

Finally, I would like to thank Birge Huisgen-Zimmermann and Steffen Koenig for their helpful remarks which led me in the right direction.
\end{ackn}

\section{Basic notions and properties}
In \cite{CPS} Cline, Parshall and Scott introduced the notion of a stratifying ideal as well as the notion of stratified and standardly stratified algebras. 
\begin{defn}
An ideal $J$ in an algebra $A$ is called \emph{stratifying} if the following conditions are satisfied.
\begin{itemize}
\item[(i)] $J = AeA$ for some idempotent $e \in A$,
\item[(ii)] Multiplication induces an isomorphism $Ae \otimes_{eAe} eA\tilde{\to} J$,
\item[(iii)] $\Tor_{n}^{eAe}(Ae,eA) = 0$ for all $n > 0$. 
\end{itemize}
It was also observed by Cline, Parshall and Scott that an ideal $J$ in $A$ is stratifying if and only if the derived functor $i_{\star}: \mathbf{D}^{+}(A/J) \to \mathbf{D}^{+}(A)$ induced by the exact inflation functor $i_{\star}: \mod A/J \to \mod A$ is a full embedding. \\
A \emph{stratification of $A$ of length $n$} is a chain \[0 = J_0 \subset J_1 \subset \dots \subset J_n = A \]
of ideals with the property that $J_i / J_{i-1}$ is a stratifying ideal in $A/J_{i-1}$. The stratification is called (left-)\emph{standard} if $J_i / J_{i-1}$ is projective (as left $A/J_{i-1}$-module).
\end{defn}
The class of algebras which we will define and deal with fits in this framework. We will make this precise in Remark \ref{vergleich}.
\begin{defn}
Let $A$ be a finite-dimensional algebra over some field $k$. Then we say that $A$ has a \emph{directed stratification of length $n$} if there exist pairwise orthogonal idempotents $e_1,\dots,e_n$ in $A$ with $\sum_{i=1}^{n} e_i = 1_A$ such that $e_i A e_j = 0$ for all $i<j$. 	
\end{defn}
One should note that we do not require the idempotents to be primitive. It is clear that every algebra admits a directed stratification of length $1$ given by its identity element, but in this case the theory we will develop will give us nothing new. 
\begin{exm}
\begin{itemize}
\item[(1)] Let $\C$ be a finite and skeletal EI-category with $n$ objects and $A = k\C$ its category algebra. Then we have a partial order defined on the set of objects of $\C$ which gives us a directed stratification of length $n$ given by the idempotents $1_{X_i}$, where $X_i, i=1,\dots,n$ are the objects of $\C$ and the numbering respects the partial order.
\item[(2)] Let $Q$ be any finite quiver without oriented cycles and $I$ any admissible ideal in $kQ$. Then $A = kQ/I$ admits a directed stratification of length $|Q_0|$ given by the primitive idempotents $\varepsilon_i$ with a suitable numbering. 
\end{itemize}
\end{exm}
As a matter of fact, one can identify the module category of an algebra $A$ with a directed stratification given by $e_1,\dots,e_n$ with the category of representations of a certain category $\A$ which we will define now.
\begin{defn}
Let $A$ be as above. Then the associated category $\A$ is defined as follows. The objects $x_1,\dots,x_n$ of $\A$ are in bijective correspondence with the idempotents $e_1,\dots,e_n$ that define our stratification and the morphisms $x_i \to x_j$ are in bijective correspondence with a $k$-basis of $e_jAe_i$. Under the assumption that $A$ is a finite-dimensional algebra, the category $\A$ is finite.
\end{defn}
By means of this definition, $A$ is the category algebra $k\A$ of $\A$.\\
In this setting it is a well-known fact due to Mitchell, that the categories $\rep_k(\A) = \Fun(\A,\mod k)$ and $\mod A$ are equivalent. For this reason, we will switch frequently between the concepts of representations and modules without any further explanation. For instance for an $A$-module $M$ we write $M(x_i)$ for its evaluation at the object $x_i$ as a functor.

\begin{Rem}\label{vergleich}
With the above characterisation of $A$ as the category algebra of $\A$, we may apply a result of Webb \cite[Proposition 2.2]{Webb} which almost immediately gives that any algebra with a directed stratification is also stratified in the sense of Cline, Parshall and Scott. Precisely, if $A$ has a directed stratification given by $e_1,\dots, e_n$, then we take $J_i = A(\sum_{l=n-i}^{n}e_l)A$. These are indeed stratifying ideals by Webbs theorem and therefore give a stratification of $A$ of length $n$. Another result of Webb \cite[Theorem 2.5]{Webb} characterizes the standardly stratified EI-category algebras to be exactly those, that are given by an EI category $\C$ in which for every morphism $\alpha: x \to y$ the group $\Stab_{\Aut(y)}(\alpha) = \Set{ \theta \in \Aut(y) | \theta\alpha = \alpha}$ has order invertible in $k$. Therefore, we may for instance take the category algebra $k\C$ of the following EI category
\[\C: 
\xymatrix{y \ar@(ul,dl)[]_{1_y}\ar[rr]^{\alpha}  && x  
\ar@(ur,ul)[]_{1_x} \ar@(dl,dr)[]_h},\; \; h^2 = 1_x, \; h\alpha = \alpha. \]
If $k$ has characteristic $2$, then the category algebra $k\C$ is not standardly stratified but it clearly admits a directed stratification given by the idempotents $1_x, 1_y$.

Hence, an algebra with a directed stratification is always stratified but in general not standardly stratified. This is an interesting point since the finitistic dimension conjecture is known to hold for standardly stratified algebras by work of \'Agoston, Happel, Luk\'acs and Unger \cite{Happeletal} while it is still open for  stratified algebras.
 \end{Rem}

We can describe the algebras $A$ which do not admit a non-trivial directed stratification (i.e.\ of length $\geq 2$) in a very convenient way by conditions that should be rather easy to check in examples. This characterisation is given by the following Proposition.

\begin{prop}\label{minimal}
Let $A$ be a finite-dimensional  $k$-algebra, where $k$ is any field. Denote by $Q$ its Gabriel quiver. Then $A$ admits a non-trivial directed stratification if and only if there exist disjoint subsets $Q_{0}'$ and $Q_{0}''$ of $Q_0$ satisfying the following conditions. \begin{itemize}
\item[$(1)$]  $Q_0 = Q_{0}' \cup Q_{0}''$,
 \item[$(2)$] For any $i \in Q_{0}'$ and $j \in Q_{0}''$ there is no path from $i$ to $j$ in $Q$.
 \end{itemize}
\end{prop}

\begin{proof}
First assume that $A$ admits a directed stratification. We may without loss of generality assume that it has length $2$ and hence is given by two idempotents $e$ and $f$ with $eAf= 0$. Denote by $\Gamma$ the Gabriel quiver of $eAe$ and by $\Gamma'$ the Gabriel quiver of $fAf$. Then we set $Q_{0}' = \Gamma_0$ and $Q_{0}'' = \Gamma'_{0}$. Since we have $1 = e + f$ condition (1) is satisfied and $eAf = 0$ implies the second condition. \\
For the converse implication assume that conditions (1) and (2) hold. Then put $e = \sum_{i \in Q_0'} e_i$ and $f = 1 - e = \sum_{j \in Q_0''} e_j$. By definition we have $ 1 = e + f$ and $e A f =  0$ follows from condition (2).
\end{proof}
Our main result in this chapter will then reduce the finitistic dimension conjecture to exactly the class of algebras mentioned in the proposition above.

The following result describes the simple and the projective $A$-modules if $A$ admits a directed stratification. It is completely analogous to the one given by L\"uck in \cite{Lueck} for EI-category algebras and also follows from work of Auslander in \cite{Auslander}.
\begin{prop}
Let $A$ be an algebra with a directed stratification given by idempotents $e_1,\dots,e_n$. Then for every simple $A$-module $S$ one has $e_i S \neq 0$ for exactly one $e_i$. In other words, as a representation $S$ is supported on exactly one object $X_i$ and $S(X_i)$ is a simple $e_i A e_i$-module. Their projective covers (i.e.\ all the indecomposable projective $A$-modules) are of the form $A e$ for some primitive idempotent $e \in e_i A e_i$ for some $i\in \Set{1,\dots,n}$.
\end{prop}

\begin{proof}
The assertion on the indecomposable projective modules is obvious since $1 = \sum_{i=1}^{n}e_i$. Then we decompose every $e_i$ and infer that the summands have to be in $e_i A e_i$.

Let $S$ be a simple $A$-module and choose $e_i$ with $e_i S \neq 0$. Then consider the submodule $U$ of $S$ generated by $e_i S$. Since the idempotents $e_1,\dots,e_n$ define a directed stratification we have $e_j U = 0$ whenever $j < i$. Let $N$ be the submodule of $S$ generated by all the $e_j S$ with $j > i$. Again,  since $A$ has a directed stratification, it follows that  $e_i N = 0$, which, together with the fact that $S$ is simple, implies that $N = 0$ and therefore $e_j S = 0$ for every $j \neq i$. If $e_i S$ would not be a simple $e_i A e_i$-module, then $S = e_i S$ would have a non-trivial submodule since it is itself an $e_i A e_i$-module. 
\end{proof}
With this Proposition it is natural to use the same notation as for EI-category algebras and denote the simple $A$-modules by $S_{x,V}$ where $x$ is an object of $\A$ and $V$ a simple $e_x A e_x$-module (here the idempotent $e_x$ corresponds to the object $x$) and to let $P_{x,V}$ denote the projective cover of $S_{x,V}$.

The main tool to understand the structure of projective resolutions of modules over algebras with a directed stratifications will be the use of restriction functors for EI-categories. We refer to \cite{Ich} for some notational conventions. The crucial point is that the whole theory for EI-categories can be carried over with slightly more complicated proofs. The following definition is completely analogous to the one for EI-categories.
\begin{defn} Let $A$ be an algebra with a directed stratification and let $\A$ be the associated finite category. 
\begin{itemize}
\item[(1)] Let $x$ be an object in $\A$. Then we define $\A_{\leq x}$ to be the full subcategory of $\A$ consisting of all objects $y \in \Ob\A$ with $\A(y,x) \neq \emptyset$. Similarly we define $\A_{\geq x}$.
\item[(2)] An ideal in $\A$ is a full subcategory $\B$ of $\A$ such that for any object $x$ in $\B$ we have that $\A_{\leq x} \subseteq \B$. 
\item[(3)] Let $M$ be an $A$-module. The $M$-minimal objects are the objects $x \in \Ob\A$ such that $M(x) \neq \emptyset$ and for any $y \in \Ob\A$ with  $\A(y,x)\neq \emptyset$ one has $M(y) = 0$.
\item[(4)] Let $M$ again be an $A$-module. We put $\A_{M}$ to be the full subcategory consisting of all $y \in \Ob\A$ with $\A(x,y) \neq \emptyset$ for some $M$-minimal object $x$ in $\A$.\\
\end{itemize}
\end{defn}

\begin{prop}\label{ideals}
Let $A$ be an algebra with a directed stratification, $\A$ the associated category and $\B$ an ideal in $\A$. Then the restriction $\downarrow^{\A}_{\B}$ preserves projectives.
\end{prop}
\begin{proof}
We construct an exact right adjoint $F:\rep \B \to \rep \A$ of the restriction functor in the following way. For $\M \in \rep \B$ and any morphism $f: \M \to \N$ in $\rep\B$  let 
\[F\M (x) = \begin{cases} \M (x) \text{ if } x \in \Ob \B, \\ 0 \;\;\;\;\;\;\; \text{ otherwise, } \end{cases} \; F(f)_x = \begin{cases} f_x \;  \text{ if }x \in \Ob \B, \\ 0 \;\;\;\; \text{otherwise.} \ \end{cases}\] This defines an exact functor. Now let $\M \in \rep\A$ and $\N \in \rep\B$. Then we define a morphism $\Psi: \Hom_{\rep\B}(\M\downarrow^{\A}_{\B},\N) \to \Hom_{\rep\A}(\M,F\N)$ via \[\Psi(f)_x = \begin{cases} f_x \; \text{ for }x \in \Ob\B, \\ 0 \;\;\;\; \text{otherwise.} \end{cases}\] Thanks to $\B$ being an ideal, this gives a $k$-linear map which is easily seen to be an isomorphism. Therefore, we get that $F$ is the desired exact right adjoint of the restriction.
\end{proof}
\section{Projective resolutions and the main result}
In this section we will analyse the structure of projective resolutions for modules over algebras with a directed stratification. It turns out that they can be described in the same fashion as the ones for modules over EI-category algebras, only the proofs become a little bit more involved.
\begin{thm}
Let $A$ be an algebra with a directed stratification and $\A$ the associated category. Let $M$ be an $A$-module and $P=P_{M}$ its projective cover. Then $P$ is supported on $\A_{M}$ and for any $M$-minimal object $x$ in $\A$ the module $P(x)$ is a projective cover of $M(x)$ as an $e_x A e_x$-module.
\end{thm}
\begin{proof}
\begin{enumerate}
\item[(i)] Clearly, we have $P = \bigoplus_{y,U} P_{y,U}$ for some objects $y$ in $\A$ and simple $e_y A e_y$-modules $U$. What we have to show is that no $y'$ with $y' \notin \A_{M}$ appears in that direct sum. Let us assume the contrary and suppose that there is an object $y'$ with $\A(y',x) \neq 0$ for an $M$-minimal object $x$ that appears in the direct sum decomposition of $P$. Then, for any $x$ with $\A(y',x) \neq 0$ and any $f \in \A(y',x)$ the following diagram has to commute
\[\begin{xy}\xymatrix{ P_{y',U'}(x) \ar[rr]^{\pi_x} & & M(x) \\ & &   \\ P_{y',U'}(y') \ar[uu]^{P_{y',U'}(f)}\ar[rr]^{\pi_{y'}} && M(y') = 0, \ar[uu] },\end{xy}\]
where $\pi: P \to M$ is the defining essential epimorphism. By the characterization of the projective $A$-modules we have $\sum_{f} \Im P_{y',U'}(f) = P_{y',U'}(x)$, which gives that $\pi_{x} = 0$ (for any such $x$). This is a contradiction to the minimality of $P$ and we have proven the first assertion.
\item[(ii)]
Since $\A_{\leq x}$ is an ideal and $\A_{\leq x} \cap \A_{\M} = \lbrace x \rbrace$, it follows from Proposition \ref{ideals} that $P(x)$ is projective. Thus, we only have to show that $P(x)$ is the projective cover of $M(x)$. \\
Suppose $P(x)$ would not be the projective cover of $M(x)$. Then $P(x) = Q' \oplus Q''$ where $Q', Q''$ are projective and $Q'$ is the projective cover of $M(x)$, whereas $\pi_x (Q'') = 0$. Since $x$ is $M$-minimal, we have that $Q' = P'(x)$ and $Q'' = P''(x)$ for some projective $A$-modules $P'$ and $P''$. Denote again by $\pi$ the defining essential epimorphism $P \to M$. Now, using that $P$ is supported on $\A_{M}$ and a similar diagram as in the first part of our proof, we get that $\pi (P'') = 0$ which contradicts the minimality of $P$.\qedhere
\end{enumerate}
\end{proof}
\begin{cor}
Let $A$ be an algebra with a directed stratification and $\A$ the associated category. Suppose that $M$ is an $A$-module and $\P$ a minimal projective resolution of $M$. Then for any $M$-minimal object $x$ in $\A$ we have that $\P(x)$ is a minimal projective resolution of $M(x)$ as an $e_x A e_x$ module.
\end{cor} 
With this characterisation of projective resolutions of $A$-modules we get the following theorem, which, roughly speaking, states that the finitistic dimension of an algebra with a directed stratification is determined by the finitistic dimension of the strata.
\begin{thm} Let $A$ be an algebra with a directed stratification and $\A$ the associated category.
\begin{enumerate}
\item[$(1)$]$A$ has finite finitistic dimension if and only if $e_x A e_x$ has finite finitistic dimension for any object $x$ of $\A$. In this case $\findim A \leq \sum_{x\in\Ob\A} \findim e_x A e_x + |\Ob\A|-1 $.
\item[$(2)$] $A$ is of finite global dimension if and only if $e_x A e_x$ is of finite global dimension for any object $x$ of $\A$. In this case $\gldim A \leq \sum_{x\in\Ob\A} \gldim e_x A e_x + |\Ob\A|-1$.\qedhere
\end{enumerate}
\end{thm}
\begin{proof}
We only prove part (1). The proof of (2) is completely analogous.

Let $e_1,\dots,e_n$ be the idempotents, that give the directed stratification of $A$. First suppose that the algebra $e_i A e_i$ has infinite finitistic dimension for some $i=1,\dots,n$. Then there exists an indecomposable $e_i A e_i$-module $N$ of projective dimension at least $d$ for any natural number $d\geq 1$. Let $\P$ be a minimal projective resolution of $N$ as an $A$-module. By definition, the object $x$ of $\A$ corresponding to $e_i$ is $N$-minimal. Hence, $\P(x)$ is a minimal projective resolution of $N$ as an $e_i A e_i$-module, which therefore is of length at least $d$. Thus, we infer that $N$ regarded as an $A$-module has projective dimension at least $d$ and $A$ has infinite finitistic dimension.

Now assume that for $i=1,\dots,n$ every $e_i A e_i$ has finite finitistic dimension and denote by $x_i$ the object in $\A$ corresponding to $e_i$. Let $M$ be an $A$-module of finite projective dimension. We consider a minimal projective resolution of $M$:
$$\P : 0 \to P^m \to P^{m-1} \to \cdots \to P^1 \to P^0 \to M \to 0.$$
The object $x_1 \in \Ob\A$ is $M$-minimal for every $A$-module $M$. Therefore, $\P(x_1)$ is a minimal projective resolution of $M(x_1)$ as $e_1 A e_1$-module. Hence, $P^d(x_1) = 0$ for all $d > \findim e_1 A e_1$. Denote by $s$ the largest integer for which $P^s(x) \neq 0$. Then, for any $d > s$, the module $P^d$ is supported on $\Ob\A \setminus \lbrace x_1 \rbrace$ and the object $x_2$ is $N$-minimal, where $N = \ker (P^s \to P^{s-1})$. With the same argument as above we infer that $P^{s+d}(x_2) = 0$ for all $d > \findim e_2 A e_2$. Now the claimed inequality follows by induction. \qedhere
\end{proof}
The following corollary is equivalent to the theorem from above.
\begin{cor}\label{twoip}
Let $A$ be a finite-dimensional $k$-algebra with a directed stratification of length $2$ given by idempotents $e,f \in A$ (i.e. $1=e+f$ and $eAf= 0$). Then the following statements hold.
\begin{enumerate}
\item[$(1)$] $A$ has finite finitistic dimension if and only if $eAe$ and $fAf$ have finite finitistic dimension. In this case $\findim A \leq \findim eAe + \findim fAf + 1$ .
\item[$(2)$] $A$ has finite global dimension if and only if $eAe$ and $fAf$ have finite global dimension. In this case $\gldim A \leq \gldim eAe + \gldim fAf + 1$.
\end{enumerate}
\end{cor}

\begin{Rem}\begin{itemize}
\item[(i)] As mentioned in the introduction, Corollary \ref{twoip} has already been obtained by Fossum, Griffith and Reiten and it also implies our theorem by using induction. Nevertheless, our proof is different and provides us with interesting new information about the structure of projective resolutions of $A$-modules, if $A$ has a directed stratification. Furthermore we will see that the iterated version of Corollary \ref{twoip} can easily be applied in examples that have not been studied so far.
\item[(ii)]
Another immediate corollary of the theorem is the following well-known fact: Let $Q$ be a finite quiver without oriented cycles and $I$ any admissible ideal in $kQ$. Then the algebra $kQ/I$ has finite global dimension. The theorem applies to this setting in the way that we take the natural directed stratification given by the primitive idempotents $\varepsilon_{i}$ in a suitable numbering. Then every stratum is just the ground field $k$ which has finite global dimension.
\end{itemize}
\end{Rem}
One interpretation of our result from above is to understand it as a technique to reduce the finitistic dimension conjecture to a smaller class of algebras, namely those that do not admit a non-trivial directed stratification. These algebras have been characterised in Proposition \ref{minimal}. 

\section{Relation to known results and examples} 
\subsection{Relation to recollements}
In \cite{Happel} Happel developed a reduction technique for the finitistic dimension conjecture (and other homological conjectures) using recollements of bounded derived categories.

\begin{thm}[Happel,\cite{Happel}]
Let $A$ be a finite-dimensional algebra and assume that $\mathbf{D}^b(A)$ has a recollement relative to $\mathbf{D}^b(A')$ and $\mathbf{D}^b(A'')$ for some finite-dimensional algebras $A'$ and $A''$. Then $\findim A < \infty$ if and only if $\findim A' < \infty$ and $\findim A'' < \infty$. 
\end{thm}
The structure of this result is similar to that of Corollary \ref{twoip}. Therefore, it is a natural question to ask if the two reduction techniques are equivalent. To see that this is not the case we have to translate the setting of algebras with a directed stratification into the language of triangulated categories. Clearly, the triangulated categories that will appear are the (bounded) derived module categories of the algebra itself and the algebras $e_i A e_i$.

Let $A$ be a finite-dimensional algebra with a directed stratification of length $2$ given by idempotents $e$ and $f$ with $eAf = 0$. Then $J := AfA$ is a stratifying ideal of $A$. With $B= A/J$ it is clear that $B \cong eAe$. Following \cite{CPS}, we have partial recollement diagrams
\[ \begin{xy}\xymatrix{ \mathbf{D}^{+}(eAe) \ar@<4pt>[r]^{i_{\star}} & \mathbf{D}^{+}(A) \ar@<4pt>[l]^{i^{!}} \ar@<4pt>[r]^{j^{\star}} & 
\mathbf{D}^{+}(fAf) \ar@<4pt>[l]^{j_{\star}} }\end{xy}\]
\[ \begin{xy}\xymatrix{ \mathbf{D}^{-}(eAe) \ar@<-4pt>[r]_{i_{\star}} & \mathbf{D}^{-}(A) \ar@<-4pt>[l]_{i^{\star}} \ar@<-4pt>[r]_{j^{\star}} & \mathbf{D}^{-}(fAf) .\ar@<-4pt>[l]_{j_{!}} }\end{xy}\]
If all the algebras involved have finite global dimension we get a full recollement of the bounded derived categories. Cline, Parshall and Scott also proved that $A$ has finite global dimension if and only if both $fAf$ and $eAe$ have finite global dimension in this situation, which is exactly the result we obtained with our description of the projective resolutions of modules for $A$. The situation for the finitistic dimension is more complicated. The following theorem provides a criterion for the above diagrams to become full recollement diagrams for the bounded derived categories.
\begin{thm}[Cline, Parshall, Scott \cite{CPS3}]
Let $A$ be a ring, $J$ an ideal in $A$ and $B = A/J$. The functor $i_{!} = i_{\star} : \mathbf{D}^b(B) \to \mathbf{D}^b(A)$ has a right adjoint $i^{!}$ satisfying $i_{!}i_{!} = \id_{\mathbf{D}^b(B)}$ if and only if
\begin{itemize}
\item[$(1)$] $\Ext^n_A(B_A,F) = 0$ for all $n >0$ and all free right $B$-modules $F$ and
\item[$(2)$]$ \pd B_A < \infty$.
\end{itemize}
\end{thm}
We will now present an example of an algebra with a directed stratification in which condition (2) is not satisfied. Consider the following EI-category $\C$ (and its category algebra $A = k\C$) in characteristic $2$:
\[\C: 
\xymatrix{y \ar@(ur,ul)[]_{1_y} \ar@(dl,dr)[]_{g}\ar[rr]^{\alpha}  && x  
\ar@(ur,ul)[]_{1_x} \ar@(dl,dr)[]_h},\; \; g^2 = 1_y, \; h^2 = 1_x, \; h\alpha = \alpha = \alpha g. \]
In this example we choose $f = 1_x$ and $ e = 1_y$. The stratifying ideal is $J = AfA = \langle 1_x, h, \alpha \rangle_k$. Thus, the algebra $eAe$ is just the group algebra $k\Aut(y)$. As an $A$-module, or as a representation of $\C$, the $A$-module $B$ is
\[\xymatrix{k^2 \ar@(ur,ul)[]_{(1)} \ar@(dl,dr)[]_{M}\ar[rr]^{0}  && 0 
\ar@(ur,ul)[]_{0} \ar@(dl,dr)[]_0}, \mbox{ where } M = (\begin{smallmatrix} 0 & 1 \\ 1 & 0 \end{smallmatrix}). \]
The projective cover $P_1 = P_B$ of $B$ as an $A$ module is
\[\xymatrix{k^2 \ar@(ur,ul)[]_{(1)} \ar@(dl,dr)[]_{M}\ar[rr]^{0}  && k
\ar@(ur,ul)[]_{(1)} \ar@(dl,dr)[]_{(1)}}, \mbox{ again with } M = (\begin{smallmatrix} 0 & 1 \\ 1 & 0 \end{smallmatrix}). \]
The kernel $K$ of the essential epimorphism $P_1 \to B$ is the representation
\[\xymatrix{0 \ar@(ur,ul)[]_{(0)} \ar@(dl,dr)[]_{(0)}\ar[rr]^{0}  && k 
\ar@(ur,ul)[]_{(1)} \ar@(dl,dr)[]_{(1)}} \]
and we denote by $P_2$ the projective module
\[\xymatrix{0 \ar@(ur,ul)[]_{(0)} \ar@(dl,dr)[]_{(0)}\ar[rr]^{0}  && k^2 
\ar@(ur,ul)[]_{1} \ar@(dl,dr)[]_{M}}, \mbox{ with } M = (\begin{smallmatrix} 0 & 1 \\ 1 & 0 \end{smallmatrix}). \]
Finally, the minimal projective resolution of $B$ as an $A$ module looks as follows
\[\dots \to P_2 \to P_2 \to P_2 \to P_1 \to B \to 0. \]
In particular, the projective dimension of $B$ as an $A$ module is infinite. Therefore, we do not have a recollement of $A$ relative to $eAe$ and $fAf$, which means that we are not in the position to apply Happels result.  However, with our result on algebras with a directed stratification it is obvious that $A$ has finite finitistic dimension since the group algebras $k\Aut(x) = fAf$ and $k\Aut(y)=eAe$ have this property.

\subsection{A non-trivial example}
To prove the finiteness of the finitistic dimension of an explicitly given algebra,  there are at the moment (to the best knowledge of the author) four important classes of algebras where the finiteness of $\findim$ is known. First of all for algebras with radical cube zero and for monomial relation algebras Huisgen-Zimmermann showed that $\findim$ is always finite. Those classes are very easy to detect, if the algebra in question is given by its quiver and relations. Igusa and Todorov proved that the finitistic dimension of an algebra $\Lambda$ is finite if $\Lambda$ has representation dimension at most $3$. This result is often difficult to apply to examples. For instance, it is not known how to calculate (or give an upper bound) for the representation dimension of a group algebra of an arbitrary finite group.

The fourth class of algebras where the finitistic dimension conjecture is known to hold true is the class of algebras $\Lambda$ for which the category $\P^{\infty}(\mod\Lambda)$ of modules of finite projective dimension is contravariantly finite in $\mod \Lambda$. This result is due to Auslander and Reiten. 

The property of $\P^{\infty}$ being contravariantly finite in $\mod A$ has been investigated by Happel and Huisgen-Zimmermann in \cite{Birge-Happel} and they observed that this property is rather 'unstable'. According to their work, an easily recognisable situation in which $\P^{\infty}$ is not contravariantly finite in $\mod kQ/I$ is the following one: $p$ is an arrow $e_1 \to e_2$, $q \in kQ\setminus I$ a path from $e_1$ to $e_2$ of positive length, different from $p$ such that $\rad(\Lambda)p = 0 = q\rad(\Lambda)$, and $\pd (\Lambda q) < \infty$ while $\pd (\Lambda e_2 / \rad(\Lambda)e_2) = \infty$.
\begin{exm}
Let $Q$ be the following quiver:
\[\begin{xy}\xymatrix{ & & 3\ar[dr]^{\delta_2} &  \\1 \ar@<4pt>[r]^{\alpha} \ar@<-4pt>[r]_{\beta} & 2 \ar@(ul,ur)[]^{\gamma}   \ar@<-2pt>[ur]^{\delta_1} \ar[dr]_{\varepsilon_1} & & 5 \ar@(u,r)[]^{\rho}   \\    &  & 4 \ar[ur]_{\varepsilon_2} &   }\end{xy}\]
Then let $A = kQ/I$ where $I$ is the ideal generated by $$\Set{\gamma^2, \gamma\beta, \varepsilon_1\beta, \delta_1\beta,\varepsilon_1\gamma, \delta_1\gamma, \varepsilon_2\varepsilon_1 - \delta_2\delta_1, \rho\varepsilon_2, \rho\delta_2, \rho^5 }.$$
$A$ is by definition not monomial and does not satisfy $\rad^3(A) = 0$. Furthermore this algebra doesn't have the property that $\P^{\infty}(\mod A)$ is contravariantly finite in $\mod A$. To see this, we put $q = \alpha$ and $p = \beta$. Then $\rad(A) p = 0 = q \rad(A)$ and $ A q = A \alpha = A e_2$ is a projective module. The module $Ae_2 / \rad(A)e_2$ is of infinite projective dimension, since it is (as a representation) non-zero only on the vertex $2$ and as an $e_2 A e_2$-module it has infinite projective dimension. Here we use that the idempotents $e_1,\dots,e_5$ give a directed stratification of $A$ and our results from the previous section.

The finiteness of $\findim$ for $A$ follows on the one hand very easily from the fact that $e_1,\dots,e_5$ give a directed stratification of $A$ and on the other hand by considering $e = e_1 + \dots + e_4$ and $f = e_5$. These two idempotents give a directed stratification of $A$ of length $2$ and $eAe$ as well as $fAf$ have finite finitistic dimension, since they are monomial relation algebras. Here we notice that the iterated version of the result of Fossum, Griffith and Reiten, which we deduced with our methods, gives the finiteness of $\findim A$ almost immediately, while for a directed stratification of length two one has to be more careful and use non-trivial results of other authors.

It is also interesting to point out that, for this example, there is no reasonable recollement-situation in sight which would give us the finiteness of $\findim A$ immediately. For instance, if we take $e$ and $f$ as above, then the stratifying ideal is $J = AfA$ and $B:=A/J = eAe$. In this case we don not get a recollement of $\mathbf{D}^b(A)$ relative to $\mathbf{D}^b(eAe)$ and $\mathbf{D}^b(fAf)$ because $B$ as an $A$-module has the simple module $Ae_2 / \rad(A)e_2$ as a summand and is therefore of infinite projective dimension.

To sum up, we have constructed an example of an algebra $A$ which is not monomial, does not satisfy $\rad^3(A)=0$, does not have the property that $\P^{\infty}(\mod A)$ is contravariantly finite in $\mod A$ and does not admit a recollement-situation which gives the finiteness of $\findim$. Nevertheless, the finiteness of $\findim A$ follows immediately from our theorem because we have a directed stratification with strata that are of finite finitistic dimension. What we have not done is to calculate the representation dimension of $A$ in this example, but in general the calculation of $\repdim$ is a very hard task. Here it is (at least for the author) not obvious that $A$ has representation dimension at most $3$ and if it would be the case, then one could imagine how one could construct arbitrarily complicated examples to which our theorem applies and where one cannot calculate the representation dimension. 
\end{exm}

\nocite{*}
\bibliographystyle{plain}
\bibliography{Refs_stratifications}
\end{document}